\documentclass[11pt,leqno,oneside,letterpaper]{amsart} 
\usepackage{amsfonts}
\usepackage{amsmath}
\usepackage{amssymb}
\usepackage{latexsym}
\usepackage[colorlinks,linkcolor=blue,citecolor=blue]{hyperref} 
\usepackage{marginnote}
\usepackage[all]{xy}

\newtheorem{Theorem}{Theorem}[section]
\newtheorem{cor}[Theorem]{Corollary}
\newtheorem{lemma}[Theorem]{Lemma}
\newtheorem{prop}[Theorem]{Proposition}

\newtheorem{claim}[Theorem]{Claim}

\newtheorem{conjecture}[Theorem]{Conjecture}

\def\qed{{\hspace{2mm}{\hfill \small $\Box$}}}

\def\C{{\mathbb C}}

\def\Z{{\mathbb Z}}
\def\Q{{\mathbb Q}}

\def\F{{\mathbb F}}

\begin{document}

\title{On the Tits alternative for $PD(3)$ groups}

\author{Michel Boileau}  
\thanks{Michel Boileau was partially supported by ANR projects 12-BS01-0003-01 and 12-BS01-0004-01.}
\address{Aix Marseille Univ, CNRS, Centrale Marseille, I2M, Marseille, France, 39, rue F. Joliot Curie, 13453 Marseille Cedex 13}
\email{michel.boileau@cmi.univ-mrs.fr }
 
\author[Steven Boyer]{Steven Boyer}
\thanks{Steven Boyer was partially supported by NSERC grant RGPIN 9446-2013}
\address{D\'epartement de Math\'ematiques, Universit\'e du Qu\'ebec \`a Montr\'eal, 201 avenue du Pr\'esident-Kennedy, Montr\'eal, QC H2X 3Y7.}
\email{boyer.steven@uqam.ca}
\urladdr{http://www.cirget.uqam.ca/boyer/boyer.html}

\maketitle

\begin{center}
\today 
\end{center} 

\begin{abstract}
We prove the Tits alternative for an almost coherent $PD(3)$ group which is not virtually properly locally cyclic. In particular, we show that an almost coherent $PD(3)$ group which cannot be
generated by less than four elements always contains a rank 2 free group. 
\end{abstract}

\section{Introduction}

Recent years have seen spectacular progress in our understanding of the algebraic properties of the fundamental groups of $3$-manifolds due mainly to Perelman's geometrisation theorem (cf. \cite{KL}, \cite{BBBMP}) and the work of Agol \cite{Ag} and Wise \cite{Wi1}, \cite{Wi2}. We know that hyperbolic $3$-manifold groups are subgroup separable, which means that their finitely generated subgroups are the intersection of subgroups of finite index. (Equivalently, finitely generated subgroups are closed in the profinite topology.) Moreover, except for graph manifolds which admit no Riemannian metrics of strictly negative sectional curvature, each closed manifold virtually fibers over the circle and thus its fundamental group has a finite index subgroup which is an extension of $\Z$ by a surface group, see \cite{Ag}, \cite{Li}, \cite{PW}, \cite{Wi1}. Necessary and sufficient conditions for a given group to be isomorphic to a closed 3-manifold group have been given in terms of group presentations (\cite{Go}, \cite{Tu}), however, no intrinsic algebraic characterisation of $3$-manifold groups is currently known.

Simple constructions show that if $n \geq 4$, each finitely presented group is the fundamental group of a smooth, closed, $n$-dimensional manifold. However, the constructions do not produce aspherical manifolds in general, so the (co)homology of the constructed manifolds does not necessarily coincide with that of the group. When $G$ is the fundamental group of a closed, aspherical $n$-manifold, its homology and cohomology with coefficients in a $\mathbb{Z}G$-module satisfy a form of Poincar\'e duality, and a group with this property is called an {\it $n$-dimensional Poincar\'e duality group} or, for short, a {\it $PD(n)$ group}. Equivalently, $G$ is a $PD(n)$ group if the following two conditions hold (see \cite{BiEc2}, \cite{Bi}, \cite[Chapter VIII]{Br}):
\vspace{-.2cm} 
\begin{itemize}
\item there is a projective $\mathbb{Z}[G]$-resolution of the trivial $\mathbb{Z}[G]$-module $\mathbb{Z}$ which is finitely generated in each dimension and 0 in all but finitely many dimensions (i.e.  $G$ is of type $FP$); 
\vspace{.2cm} \item $H^n(G; {\Z}G) \cong \Z$ and $H^i(G; {\Z}G) = \{0\}$ for all $i \not = n$. 
\end{itemize}
The first condition implies that $PD(n)$ groups are torsion-free. The latter implies that for $n \geq 2$, $PD(n)$ groups are $1$-ended and therefore indecomposable by Stallings theorem (cf. \cite[4.A.6.6]{Sta2}). 

Finite index subgroups of $PD(n)$ groups are $PD(n)$.

A $PD(n)$ group $G$ is said {\it orientable} if the action of $G$ on $H^n(G; {\Z}G)$ is trivial, and non-orientable otherwise. Each $PD(n)$ group contains an orientable $PD(n)$ group of index $1$ or $2$.

Wall asked in the 1960s whether this duality characterises the fundamental groups of closed aspherical $n$-dimensional manifold. But Davis produced examples of $PD(n)$-groups which are not finitely presentable, and so cannot be  the fundamental group of a closed, aspherical $n$-dimensional manifold \cite[Theorem C]{Dav}, for each $n \geq 4$. This leads to the following conjecture:

\begin{conjecture} {\rm (Wall)} \label{conj: Wall}
A finitely presented $PD(n)$ group is isomorphic to the fundamental group of a closed, aspherical $n$-dimensional manifold. 
\end{conjecture}

Non-trivial arguments of Eckmann, Linnell and M\"{u}ller (see \cite{Ec}) verified the conjecture for $PD(2)$ groups, and though the $PD(3)$ case remains completely open, progress has been made. For instance, Perelman's geometrisation theorem implies that a $PD(3)$ group is isomorphic to a $3$-manifold group if and only if it is virtually a $3$-manifold group. (See \cite[Theorem 5.1]{GMW}.) 

A group $G$ is called $FP_n$ if there is a projective $\mathbb{Z}[G]$-resolution of the trivial $\mathbb{Z}[G]$-module $\mathbb Z$ which is finitely generated in degrees $n$ or less. The condition $FP_1$ is equivalent to being finitely generated. The condition $FP_2$ is implied by being finitely presented and a group satisfying this condition is called {\it almost finitely presented}.  Since $PD(3)$ groups are $FP$, they are finitely generated and almost finitely presented. 
 
Hillman verified Conjecture \ref{conj: Wall} for $PD(3)$ groups of positive first Betti number which contain non-trivial, normal cyclic subgroups (\cite{Hil1}). Bowditch extended this to general $PD(3)$ groups containing non-trivial, normal cyclic subgroups (\cite{Bow}). This, together with results of Hillman (\cite{Hil1}, \cite{Hil2}, \cite{Hil3}) and Thomas (\cite{Tho1}), implies that the conjecture holds for $PD(3)$ groups containing a non-trivial, infinite index normal subgroup which is almost finitely presented.

A group is \emph{coherent} if each of its finitely generated subgroups is finitely presentable; it is \emph{ almost coherent} if
each of its finitely generated subgroups is almost finitely presented (i.e. $FP_2$). Scott's compact core theorem \cite{Sco1} shows that $3$-manifold groups are coherent. However, it is an open question whether a $PD(3)$ group is coherent or even almost coherent. Moreover, there are examples of closed hyperbolic $4$-manifolds whose fundamental groups are not coherent \cite{BoMe}, \cite{Pot}.

Dunwoody and Swenson proved a version of the torus theorem \cite{DS}: an orientable $PD(3)$ group is either isomorphic to the fundamental group of a Seifert $3$-manifold, or it splits over a 
$\mathbb Z \oplus \mathbb Z$ subgroup, or it is atoroidal (i.e.~does not contain any $\mathbb Z \oplus \mathbb Z$ subgroup). This was generalised in F.~Castel's thesis \cite{Cas}, where Scott-Swarup regular neighborhood theory (\cite{ScSw}) was used to show that an orientable $PD(3)$ group admits a canonical JSJ splitting along free abelian groups of rank $2$ with vertex groups isomorphic to either Seifert manifold groups or atoroidal $PD(3)$ pairs in the sense of \cite[Definition 1(2)]{Cas}. This splitting is analogous to the one induced on the fundamental group of closed orientable aspherical $3$-manifolds by their  geometric decomposition. For finitely presented $PD(3)$ groups it can be deduced from Dunwoody and Swenson  JSJ-decomposition theorem \cite{DS}. See \cite[Theorem 10.8]{Wa1} or  \cite[Theorem 4.2]{Wa2}.

These results show that there is a strong correlation between the algebraic properties of $PD(3)$ groups and those of $3$-manifold groups. The goal of this paper is to further enhance this correlation by studying the existence of non-abelian free groups.

A well-known theorem of Tits states that a finitely generated linear group contains a non-abelian free group or is virtually solvable (\cite{Ti}) and we say that a group $G$ satisfies the {\it Tits alternative} if each of its finitely generated subgroups verifies this dichotomy. Geometrisation shows that $3$-manifold groups are residually finite and satisfy the Tits alternative. In fact, more recent works by Agol and Wise showed that in most cases, $3$-manifold groups are linear.

To study the Tits alternative for  $PD(3)$ groups we cannot rely on the geometry, but use, rather, algebraic methods, namely the mod $p$ cohomology of groups and the theory of pro-$p$ groups through the works of Lubotzky-Mann \cite{LM} and Shalen-Wagreich \cite{SW}. For our purposes, the alternative and weaker approach of Segal \cite{Seg} to Lubotzky's linearity criterion \cite{Lu} is sufficient and allows us to avoid the use of Lazard's theory of $p$-adic analytic groups \cite{La}.  Homological methods have already been used to approach the Tits alternative for 3-manifolds groups (see \cite{KZ}, \cite{Me1}, \cite{Pa}, \cite{SW}) and also for $PD(3)$ groups (see \cite[section 3]{Hil3}, \cite[Chapter 2]{Hil4}).

We say that a finitely generated, torsion free group is {\it properly locally cyclic} if each of its proper, finitely generated subgroups is cyclic. As mentioned above, $3$-manifold groups are residually finite and this implies that they are not virtually properly locally cyclic (cf.~Lemma \ref{lem:wild}). Mess \cite{Me2} has shown that for each $n \geq 4$, there are closed aspherical $n$-manifolds with non-residually finite fundamental group. 

An infinite, finitely generated, properly locally cyclic  group cannot satisfy the Tits alternative, unless it is virtually $\mathbb{Z}$. They are torsion free analogues of {\it Tarski monsters}:  infinite simple groups all of whose proper subgroups are infinite cyclic (see \cite{Ol}). In Lemma \ref{lem:wild} we show that locally cyclic $PD(3)$ groups are simple. 

Here is the main result of this article. It sharpens \cite[Theorem 9]{Hil3} and \cite[Theorem 2.13]{Hil4}. 

\begin{Theorem}\label{th:main}
An almost coherent $PD(3)$ group satisfies the Tits alternative if and only if  it does not contain a finite index subgroup which is properly locally cyclic. 
\end{Theorem}
\vspace{-.2cm} 

We show in Theorem \ref{th:surface} that an almost coherent $PD(3)$ group with infinite profinite completion either contains a surface group or contains a non-abelian free group. As such, it cannot be properly locally cyclic. Consequently, we obtain the following corollary which extends to $PD(3)$ groups the result of \cite{KZ}.  

\begin{cor} \label{cor:profinite}  
An almost coherent $PD(3)$ group with an infinite profinite completion satisfies the Tits alternative. 
\qed
\end{cor}

The following corollary sharpens the results of \cite{Pa} and extends them to $PD(3)$ groups.

\begin{cor} \label{cor:mod p betti}  
An almost coherent $PD(3)$ group $G$ with $\dim_{\F_p} H_1(G; \F_p)\geq 2$ for some odd prime $p$ or with  $\dim_{\F_2} H_1(G; \F_2)\geq 3$ satisfies the Tits alternative. 
\end{cor}

As already mentioned, an important consequence of geometrisation is that $3$-manifold groups are residually finite. In particular, any 3-manifold group contains a proper subgroup of finite index. One expects that the same property holds for $PD(3)$ groups, and since finite index subgroups of $PD(3)$ groups are $PD(3)$, Conjecture \ref{conj: psfi} would then imply that $PD(3)$ groups have infinite profinite completions. In particular, $PD(3)$ groups could not be virtually properly locally cyclic. 

\begin{conjecture} \label{conj: psfi}
A $PD(3)$ group always contains a proper subgroup of finite index.
\end{conjecture} 

Our remaining results determine sufficient conditions for the existence of a non-abelian free subgroup of a $PD(3)$ group. 

Scott's compact core theorem \cite{Sco1} shows that $3$-manifold groups are coherent. It would be preferable to avoid the coherence assumption in Theorem \ref{th:main}, but we need it to use an algebraic analogue of the core theorem for $PD(3)$ groups due to Kapovich and Kleiner \cite{KK}. A useful consequence of their result is that a one-ended $FP_2$ subgroup of infinite index in a  $PD(3)$ group contains a surface group. 

The {\it rank} of a group $G$, denoted $rk(G)$, is the minimal number of elements needed to generate it. The {\it Pr\"ufer-rank} of $G$, denoted $u(G)$, is the maximal rank among all finitely generated subgroups of $G$; it may be infinite. If a group $G$ has finite Pr\"ufer-rank $r$, every finitely generated subgroup of $G$ can be generated by $r$ elements. The classification of finite groups with cohomology of period $2$ and $4$ (see \cite[Chapter IV.6]{AM}) combines with Theorem \ref{th:main} to prove our next result. 

\begin{Theorem}\label{Theorem:rank} 
An almost coherent $PD(3)$ group $G$ contains a rank $2$ free group if and only if its Pr\"ufer-rank
is at least $4$.
\end{Theorem}

The following corollary of Theorem \ref{Theorem:rank} sharpens  \cite[Theorem 2.9]{SW} and extends it to $PD(3)$ groups.

\begin{cor}\label{cor:rank}   
An almost coherent $PD(3)$ group of rank $4$ or more contains a rank $2$ free group.
\qed
\end{cor}

Consideration of the $PD(3)$-group $\mathbb Z^3$ shows that the hypothesis that the rank be at least $4$ in Corollary \ref{cor:rank} is sharp. 

Corollary \ref{cor:profinite} shows that it is of interest to determine sufficient conditions on a $PD(3)$ group which guarantee that its profinite completion is infinite. Our next result gives such a condition via the group's $PSL(2,\mathbb C)$-character variety.

\begin{Theorem} \label{th:character} 
If a $PD(3)$ group $G$ has at least three distinct, strictly irreducible $PSL(2,\mathbb C)$-characters, then $\widehat{G}$ is infinite.  Moreover, if $G$ is almost coherent, it  contains a rank $2$ free group.
\end{Theorem}

\begin{cor}\label{cor:casson} 
Let $M$ be an integral homology $3$-sphere with Casson invariant $\vert \lambda(M) \vert \geq 2$. Then $\pi_1(M)$ contains a rank $2$ free group.
\qed
\end{cor} 

This corollary is known to follow from the geometrisation of 3-manifolds. In contrast, our proof is algebraic.

Here is the plan of the paper. In Section \ref{sec: ghat infinite} we study $PD(3)$ groups with infinite profinite completions and prove that when such a group is almost coherent,  
it either contains a surface group or is virtually $k$-free for each $k \geq 1$ (Theorem \ref{th:surface}). Section \ref{sec: tits} contains the proof of Theorem \ref{th:main} and Corollary \ref{cor:mod p betti} while \S \ref{sec: r2free} characterises almost coherent PD(3) groups which contain non-abelian free groups. See Theorem \ref{Theorem:rank}. Finally, we prove Theorem \ref{th:character} and Corollary \ref{cor:casson} in \S \ref{sec: psl2}.

{\bf Acknowledgements}. The authors would like to thank Ralph Strebel for his 
remarks clarifying the use of Theorem \cite[Theorem A]{BiSt} in the proof of Claim \ref{claim:positive Betti} and Jonathan Hillman for his remarks on an earlier version of the paper. The authors are also grateful to the referee for pointing out gaps in some of our arguments and for comments resulting in an improved exposition.

\section{PD(3) groups with infinite profinite completion} \label{sec: ghat infinite} 
Let $k \geq 1$ be an integer. A group is {\it $k$-free} if any subgroup generated by $k$ elements is free. For an almost coherent $PD(3)$ group with infinite profinite completion, we have the following  result  related to the Tits alternative.

\begin{Theorem}\label{th:surface}
Let $G$ be an almost coherent $PD(3)$ group. If $\widehat{G}$ is infinite, then either: 

$(1)$ $G$ contains a surface group, or 

$(2)$ $G$ is virtually $k$-free for all $k \geq 1$. 
 
\end{Theorem}

Since a $PD(3)$ group contains an index $2$ orientable $PD(3)$ group, we need only to consider orientable $PD(3)$ groups. 

The main ingredient in the proof of Theorem \ref{th:surface} is Proposition \ref{prop:Betti}, a $PD(3)$-version of a theorem of 
Lubotzky and Mann \cite{LM} concerning residually finite groups of finite rank. The proposition's proof follows the lines of \cite{LM}, except that we avoid the use of Lazard's theory of $p$-adic analytic groups \cite{La} by applying a weak linearity criterion due to Segal \cite{Seg}.

We use 
$$\beta(G) =  \dim_{\Q} H_1(G; \Q)$$ 
to denote the first Betti number of $G$ and 
$$v\beta(G) = \sup_H \{\beta(H) :  \text{$H$ has finite index in $G$}\}$$ 
to denote its virtual first Betti number.

Similarly for $p$ prime,  
$$\beta_{p}(G) =  \dim_{\F_p} H_1(G; \F_p)$$ 
denotes the (mod $p$) first Betti number of $G$ and 
$$v\beta_{p}(G) = \sup_H \{\beta_p(H) :  \text{$H$ has finite index in $G$}\}$$ 
its virtual (mod $p$) first Betti number.  

\begin{prop}\label{prop:Betti}  Let $G$ be an orientable $PD(3)$ group with infinite profinite completion $\widehat{G}$. Then either  

$(1)$ $G$ has a positive virtual first Betti number, or 

$(2)$ there is a prime number $p$ such that $v\beta_{p}(G) = \infty$.
 
\end{prop}

Before proving Proposition \ref{prop:Betti} we introduce some notation and terminology.

We say that $G$ has sectional $p$-rank less than or equal to $r$ if the $p$-Sylow subgroups of the finite quotients of $G$ have Pr\"ufer-rank at most $r$. The {\it sectional $p$-rank of $G$}, denoted $u_p(G)$, is the smallest $r$ such that $G$ has sectional $p$-rank at most $r$. 

\begin{lemma}\label{lemma:sectionalrank}
If $v\beta_{p}(G) $ is finite, then every finite index subgroup $H$ of $G$ has sectional $p$-rank $u_p(H) \leq v\beta_{p}(G)$.
\end{lemma}

\begin{proof} 
Let $H \subset G$ be a subgroup of finite index in $G$ and let $f: H \to F$ be an epimorphism onto a finite quotient of $H$. Denote by $S_p \subset  F$ the $p$-Sylow subgroup of $F$ and let $d = u(S_p)$ be its Pr\"ufer-rank.  There is a subgroup $T$ of $S_p$ with rank $rk(T) = d$ and $T /\Phi(T) \cong  (\Z/p\Z)^d$, where $\Phi(T)$ is the Frattini subgroup of $T$. 
The finite index subgroup $K = f^{-1}(T)$ of $H$ is also of finite index in $G$ and the restriction to $K$ of $f$ induces an epimorphism $K \to T \to T /\Phi(T) \cong  (\Z/p\Z)^d$. Thus $v\beta_p(G) \geq \beta_p(K) \geq d$.
\end{proof}

\begin{lemma}\label{lem:rank} Let $G$ be an orientable $PD(3)$-group and $p$ be a prime number. 
 If $v\beta_{p}(G)$ is finite, then $v\beta_{p}(G) \leq 3$ and 
any finite index subgroup $H$ of $G$ has sectional $p$-rank $u_p(H) \leq 3$.
\end{lemma}

\begin{proof}  
Let $G$ be an orientable $PD(3)$-group and let $G_1$ be the kernel of the epimorphism $G \to H_1(G,\F_p)$. Lemma 1.3 of \cite{SW} is stated for closed orientable 3-manifold groups, but its proof follows from Poincar\'e duality and a homological computation based on the $5$-term Stallings' exact sequence for the (mod $p$) homology  \cite{Sta1}. Hence  the inequality  $\beta_p(G_1) \geq  \frac{1}{2}\beta_{p}(G)(\beta_{p}(G) -1)$ remains valid for an orientable $PD(3)$-group $G$. Therefore $\beta_p(G_1) > \beta_p(G$ when $\beta_p(G) \geq 4$. Proceeding inductively, it follows that $v\beta_{p}(G) = \infty$. The lemma is a consequence of this observation and Lemma \ref{lemma:sectionalrank}. 
\end{proof} 

\begin{prop}\label{prop:soluble}
Let $G$ be an orientable $PD(3)$-group such that $v\beta_{p}(G)$ is finite for every prime $p$.
Then $G$ has a finite index subgroup $G_1$ each of whose  finite quotients is solvable with Pr\"ufer rank $\leq 4$.
\end{prop}

\begin{proof}
According to Lubotzky and Mann \cite{LM}, a group $G$ with finite $2$-rank has a finite index subgroup $G_1$ all of whose finite quotients are solvable. On the other hand, Kov\'acs \cite{Kov} has shown that if $S$ is a finite soluble group all of whose Sylow subgroups have Pr\"ufer-rank at most $r$, then $S$ has Pr\"ufer rank at most $r + 1$. Proposition \ref{prop:soluble} now follows from Lemma \ref{lem:rank}. 
\end{proof}

\begin{proof}[Proof of Proposition \ref{prop:Betti}]  
Suppose that $G$ is an orientable  $PD(3)$-group with infinite profinite completion such that $v\beta_{p}(G)$ is finite for every prime $p$.
Let $G_1$ be the finite index subgroup of $G$ given by Proposition \ref{prop:soluble} and define $K$ to be the intersection of all finite index subgroups of $G_1$. 
Since $\widehat{G}_1$ has finite index in $\widehat{G}$, it is infinite, so the residually finite quotient $\Gamma = G_1/K$ is also infinite. By \cite{Seg}, $\Gamma$ is virtually nilpotent by abelian, and as an infinite, finitely generated nilpotent by abelian group has a positive virtual first Betti number, we are done.  
\end{proof}

The key ingredient of the proof of Theorem \ref{th:surface} is the following result of Kapovich and Kleiner. 

\begin{prop} \label{prop:coarse} {\rm (\cite[Corollary 1.3(2)]{KK})}   An infinite index $FP_2$ subgroup of an orientable $PD(3)$ group
either contains a surface subgroup or is free.
\end{prop}

Theorem \ref{th:surface} now follows from Claims \ref{claim:positive Betti} and \ref{claim:p-Betti} below.

\begin{claim} \label{claim:positive Betti} 
Suppose that $G$ is an almost coherent orientable $PD(3)$ group with $\widehat{G}$ infinite. If $v\beta_{p}(G)$ is finite for each prime $p$,  then $G$ contains a surface group.
\end{claim}

\begin{proof} 
By Proposition \ref{prop:Betti}, $v\beta(G) \geq 1$, so there is a finite index subgroup $H$ of $G$ which admits an epimorphism onto $\mathbb{Z}$. Since $H$ is a PD(3) group, it is of type FP by 
\cite{BiEc2}, so the Bieri-Strebel's theorem \cite[Theorem A]{BiSt} implies that $H$ splits as an HNN extension with finitely generated base and the splitting subgroups of infinite index. Since $G$ is almost coherent, the base group of the HNN extension is $FP(2)$. If $G$ does not contain a surface group, Proposition \ref{prop:coarse} implies that the base group of the HNN extension is free. But this implies that the cohomological dimension of $H$ is $2$, which is impossible since $H$ is a $PD(3)$ group.
\end{proof}

\begin{claim} \label{claim:p-Betti} 
If there is  a prime number $p$ such that  $v\beta_{p}(G) = \infty$ and  $G$ does not contain a surface group, then $G$ is virtually $k$-free for any $k \geq 1$.
\end{claim}

\begin{proof}
Let $G_k$ be a finite index subgroup of $G$ with $\beta_{p}(G_k) \geq k + 2$ and let $H \subset G_k$ be a subgroup generated by $k$ or fewer elements. 
Since $G_k$ satisfies Poincar\'e duality, the cohomological computation using the Stallings exact sequence and the (mod $p$) lower central series of $G_k$ in \cite[Proposition 1.1]{SW}  applies to show that $H$ lies in infinitely many distinct subgroups of finite index in $G_k$. Thus it has infinite index in $G_k$. Since $G_k$ does not contain a surface group and is almost coherent, Proposition \ref{prop:coarse}  implies that $H$ is free. Hence $G_k$ is $k$-free.
\end{proof}

\section{The Tits alternative} \label{sec: tits} 

In this section we prove Theorem \ref{th:main}. We suppose throughout that $G$ is an almost coherent $PD(3)$ group and will show that $G$ satisfies the Tits alternative if and only if it does not contain a finite index subgroup which is properly locally cyclic. The forward direction is straightforward: If $G$ satisfies the Tits alternative and contains a properly locally cyclic finite index subgroup $H$, then $H$ does not contain a non-abelian free group, and since $G$ is $PD(3)$ it does not contain a finite index subgroup isomorphic to $\mathbb Z$. Thus $H$ must be solvable.  But a torsion free finitely generated solvable properly locally cyclic group is easily seen to be infinite cyclic, which is impossible. 

Thus, the proof of Theorem \ref{th:main} reduces to the following proposition. 

\begin{prop}\label{prop:tits} Let $G$ be an almost coherent $PD(3)$ group which is not virtually properly locally cyclic. Then any finitely generated subgroup of $G$ contains a rank $2$ free group or is virtually solvable.
\end{prop}

\begin{proof}
As above we can assume that $G$ is an orientable $PD(3)$ group. 

First we deal with the case that the subgroup $H$ has a finite index in $G$. Then $H$ is a non-virtually properly locally cyclic almost coherent $PD(3)$ group. We claim that it contains  either a non-abelian free group or $\mathbb Z \oplus \mathbb Z$. This follows from Theorem \ref{th:surface} if $\widehat{H}$ is infinite. If $\widehat{H}$ is finite, $H$ contains a smallest finite index normal subgroup $H_0$. Then any finitely generated proper subgroup of $H_0$ is of infinite index. Since $H$ is not properly virtually locally cyclic, $H_0$ contains a finitely generated subgroup $K$ of infinite index which is not cyclic. Since $H$ is almost coherent, $K$ is $FP(2)$ and by Proposition \ref{prop:coarse} it is either free of rank $2$ or more or contains a surface group. Hence $K$ contains a rank $2$ free group or $\mathbb{Z} \oplus \mathbb{Z}$. 

Next we claim that if $H$ contains $\mathbb{Z} \oplus \mathbb{Z}$, it contains a rank $2$ free group or is virtually solvable. By the Dunwoody-Swenson torus theorem \cite{DS} and Bowditch's theorem \cite{Bow}, $H$ is either isomorphic to the fundamental group of a Seifert $3$-manifold or splits over $\mathbb Z \oplus \mathbb Z$ with finitely generated edge groups which, as base groups of $PD(3)$ pairs (\cite[Theorem 8.3]{BiEc3}), are finitely generated. Since Seifert $3$-manifold groups verify the Tits alternative, we only need to consider the case where  $H$ splits over 
$\mathbb Z \oplus \mathbb Z$. 
 
 If the splitting corresponds to an amalgamated free product $H = A *_C B$ with $C \cong \mathbb Z \oplus \mathbb Z$, then $H$ contains a rank $2$ free group or $C$ is of index $2$ in both $A$ and $B$ (\cite[Lemma 1]{BaSh}). In the latter case, it is easy to see that $H$ is solvable.
 
Suppose that the splitting corresponds to an {\it HNN} extension $H = A *_C = \langle A, t : t C_+ t^{-1} = C_- \rangle$, with $C_+ \cong C_- \cong \mathbb Z \oplus \mathbb Z$ and $(A, C_+ \cup C_-)$ a $PD(3)$ pair. If $H$ does not contain a rank 2 free group, the {\it HNN} extension must be ascending since  $A$ is finitely generated and $FP(2)$ (see \cite[Theorem 5]{Bau}). In this case either $C_+ = A$ or $C_- = A$, and hence $H$ is virtually solvable. (This case also follows from \cite[Corollary of Theorem 3]{Hil3}.) This completes the proof of the proposition when $H$ has finite index in $G$. 

Now suppose that $H$ is of infinite index. By \cite{Str}, $H$ has cohomological dimension at most $2$. It then follows from \cite[Corollary 1.3]{KK} that $H$ admits a free product decomposition $H = F * H_1 * \cdots * H_n$ where $F$ is a finitely generated free group and each $H_i$ is the base group of a $PD(3)$ pair. Therefore, if we assume that $H$ does not contain a rank $2$ free group,
 it follows that $n = 1$ and $H$ is the base group of an orientable $PD(3)$ pair. A standard cohomological computation then shows that $H$ admits an epimorphism to 
 $\mathbb{Z}$, and as $H$ is finitely generated, almost coherent and $cd(H) = 2$, \cite[Theorem 3]{Hil3} implies that $H$ is virtually solvable. 
\end{proof}

Corollary \ref{cor:mod p betti} follows from Corollary \ref{cor:profinite} and the following lemma. 

\begin{lemma}\label{lem:no wild}  
If a $PD(3)$ group $G$ has a finite profinite completion $\widehat{G}$ then $\beta_2(G) \leq 2$ and $\beta_p(G) \leq 1$ for each odd prime $p$.
\end{lemma}

\begin{proof}
The finiteness of $\widehat{G}$ implies that there is a smallest finite index subgroup $G_0$ of $G$ which is necessarily normal and contained in every finite index subgroup of $G$. It follows that  the kernel of each finite quotient of $G$ contains $G_0$ and thus each finite quotient of $G$ factorizes through $F_0  = G/G_0$. In particular $G_0$ is perfect.

\begin{claim}\label{claim:periodic} The finite group $F_0  = G/G_0$ has periodic cohomology with period  $4$.
\end{claim}

\begin{proof} Let $X$ be a $K(G, 1)$ complex. The finite cover $X_0$ of $X$ corresponding to the subgroup $G_0$ of $G$ is a $K(G_0, 1)$ complex on which $F_0$ acts freely. Since $G_0$ is a perfect $PD(3)$ group, $X_0$ is  an integral cohomology
 $3$-sphere: $H^{*}(X_0, \mathbb{Z}) \cong H^{*}(S^3, \mathbb{Z})$. 
 Then $F_0$ has periodic cohomology of period $4$ by \cite[Thm 4.8] {Swa}, see also \cite[chapter VII, Proposition 10.2]{Br}, \cite[Chapters 16.9, Appl. 4]{CE}.
\end{proof}
 
 The Suzuki-Zassenhaus classification of periodic groups (cf. \cite[Chapter IV.6]{AM}) determines a complete list of finite groups with cohomological  period $2$ or $4$. See \cite[Table 2]{GNS} for example. Since the epimorphism $G \to H_1(G, F_p)$ factorizes through $F_0$, it follows that $F_0/[F_0,F_0]$ surjects onto $H_1(G,\F_p)$, for any prime number $p$. The computation of $F_0/[F_0,F_0]$ from the table of groups with cohomological  period $2$ or $4$ shows that $\beta_2(G) \leq 2$ and $\beta_p(G) \leq 1$ for each odd prime $p$. This follows also from the fact that the Sylow $p$-groups of $F_0$ are cyclic  or generalized quaternionic 
(\cite[Chapter IV.6]{AM}, \cite[Chapter VI.9]{Br}).
\end{proof}

 \section{Characterisation of $PD(3)$ groups containing a rank 2 free subgroup} 
 \label{sec: r2free}

The goal of this section is to prove Theorem \ref{Theorem:rank}. We begin with a lemma concerning properly locally cyclic $PD(3)$ groups.

\begin{lemma}\label{lem:wild} If a $PD(3)$ group $G$ is properly locally cyclic, then 

$(1)$ $G$ has trivial profinite completion;  

$(2)$ $G$ is simple and generated by any pair of non-commuting elements.
\end{lemma}

\begin{proof} 
Assertion $(1)$ follows from the fact that a properly locally cyclic $PD(3)$ group cannot contain a proper finite index subgroup, since such a subgroup is finitely generated of cohomological dimension $3$, and thus cannot be isomorphic to $\mathbb{Z}$.

For assertion (2), note that by \cite{Tho2}, an abelian $PD(3)$ group is a Euclidean $3$-manifold group and thus cannot be properly locally cyclic. Hence we may choose non-commuting elements  $x, y \in G$. Then by our hypotheses, $G = \left\langle x, y  \right\rangle$. 

Let $N \subset G$ be a non-trivial proper normal subgroup. By $(1)$, $N$ has infinite index. Hence $N$ is locally cyclic, so abelian, and torsion free group. Thus it is isomorphic to a subgroup of the additive group $\mathbb{Q}$. (See for example \cite[Chapter VIII, Section 30]{Ku}.)  This implies that $Aut(N)$ is abelian and there is a representation $\rho: G \to Aut(N)$ induced by conjugation.
Let $K = \ker \rho$. If $K = G$, $N$ belongs to the center of $G$ and by \cite{Bow}, $G$ is a Seifert fibered 3-manifold group and hence cannot be properly locally cyclic. Therefore $K$ is a proper subgroup of $G$ and the previous argument shows that $K$ is an abelian subgroup of infinite index in $G$. Hence $G$ is a solvable group. But a solvable $PD(3)$ group is a $3$-manifold group (\cite{Tho2}) and thus cannot be properly locally cyclic. Hence $G$ is simple. 
\end{proof}

\begin{proof}[Proof of Theorem \ref{Theorem:rank}] 
If $G$ contains a rank 2 free group, it is clear that the Pr\"ufer-rank $u(G)$ is unbounded and hence at least $4$. This is the forward direction of the theorem. To prove the reverse direction, we begin with a lemma.

\begin{lemma}\label{lem:not wild}  
A $PD(3)$ group with Pr\"ufer-rank $u(G) \geq 4$ is not virtually properly locally cyclic.
\end{lemma}

\begin{proof}
We argue by contradiction. Let $G_0 < G$ be a properly locally cyclic subgroup of  finite index. Since $G_0$ is simple (Lemma \ref{lem:wild}), it is normal in $G$. The finite quotient $F_0 = G/G_0$ has  cohomological  period $2$ or $4$ (\cite[Chapters 12.11 and 16.9]{CE}). We claim that its Pr\"ufer-rank is at most $3$. To see this, note that since $F_0$ has finite cohomological period, its Sylow $p$-groups are cyclic for $p$ odd and generalized quaternionic for $p= 2$ (cf. \cite[Chapter IV.6]{AM}). Thus their Pr\"ufer-ranks are at most $2$. It then follows from \cite{Kov} that $u(F_0) \leq 3$ if $F_0$ is solvable. 

Suppose that $F_0$ is not solvable. The  classification of periodic groups with cohomological  period $2$ or $4$ (see \cite[Table 2]{GNS}) shows that when $F_0$ is non-solvable, it is isomorphic to $\mathbb{Z}/n\mathbb{Z} \times SL_2(\mathbb{F}_5)$ with $n$ prime to $120$. Such a group is the fundamental group of a spherical Seifert fibered $3$-manifold (\cite{Tho3}) and so the same is true for any of its subgroups. Since spherical $3$-manifolds have Heegaard genus $2$ or less, $u(F_0) = u(\mathbb{Z}/n\mathbb{Z} \times SL_2(\mathbb{F}_5)) \leq 2$. 

To complete the proof of Lemma \ref{lem:not wild} we need only to show that $u(G) \leq \max \{2, u(F_0 )\} $. 

Let $\pi: G \to F_0 = G/G_0$ be the quotient homomorphism and $H \subset G$ be a finitely generated subgroup. Since $G_0$ is properly locally cyclic, there are three possibilities for the group 
$H \cap G_0$:

\noindent $(1)$ $H \cap G_0 = \{1\}$. In this case $H$ is finite, since $H \cap G_0$ has finite index in $H$. Hence $H = \{1\}$, since $G$ is torsion free, so $rk(H) = 0$.

\noindent $(2)$ $H \cap G_0 \cong \mathbb{Z}$. In this case $H$ is a virtually infinite cyclic group, and hence has $2$ ends. Since $H$ is torsion free, it follows that $H  \cong \mathbb{Z}$ and thus $rk(H) = 1$, see \cite[4.A.6.5]{Sta2}.

\noindent $(3)$ $H \cap G_0 = G_0$. Then $H$ is of finite index in $G$ and is therefore a $PD(3)$ group.  Let $r = rk(\pi(H))$ and let $\{a_1, \cdots, a_r \}$ be a minimal set of generators for $\pi(H)$.
For $i = 1, \cdots, r$ take $\bar{a}_i \in H$ such that $\pi(\bar{a}_i) = a_i$ and define $K = \left\langle \bar{a}_1, \cdots, \bar{a}_r \right\rangle \subset H$. We consider the subgroup $K \cap G_0$ of $G_0$ and as above we argue according to the three possibilities:

\noindent $(i)$ If  $K \cap G_0 = \{1\}$, then $K = \{1\}$. Hence $\pi(K) = \pi(H) = \{1\}$ and $H = G_0$. Then $rk(H) \leq 2$ by Lemma \ref{lem:wild}.

\noindent $(ii)$ If $\{1\} \neq K \cap G_0 \neq G_0$, then $K \cap G_0 \cong \mathbb{Z}$. This implies that $K  \cong \mathbb{Z}$ is generated by a single element $t \in H$.
Since $G_0$ is of finite index in $H$, there exits an integer $n \geq 1$ such that $t^n$ generates $K \cap G_0$. Since a Seifert fibered 3-manifold group is not properly locally cyclic, $K \cap G_0$ cannot lie in the center of $G_0$ by \cite{Bow}. Hence there exists an element $x \in G_0 \setminus K$ which does not commute with $t^n$. Then
by Lemma \ref{lem:wild}, 
$G_0 = \left\langle x, t^n \right\rangle$, and so $G_0 \subset \left\langle x, t \right\rangle$. Since $\pi(\left\langle x, t \right\rangle) = \pi(K) = \pi(H)$, it follows that 
$H = \left\langle x, t \right\rangle$ and $rk(H) = 2$.

\noindent $(iii)$ If $K \cap G_0 = G_0$ then $G_0 \subset K$. Since $\pi(K) = \pi(H)$, $H = K$ and thus $rk(H) = r \leq u(F_0)$.

Thus for any subgroup $H$ of $G$, $rk(H) \leq \max (2, u(F_0)$. Hence $u(G) \leq \max\{2, u(F_0)\}$.
\end{proof}

Now we complete the proof of Theorem \ref{Theorem:rank}.

Lemma \ref{lem:not wild} and Theorem \ref{th:main} imply that $G$ satisfies the Tits alternative. In particular it contains a rank 2 free group or is virtually solvable. 
By \cite{AuJo}, see \cite{Tho2}, a virtually solvable $PD(3)$ group is the fundamental group of a closed $3$-manifold $M$ which is geometric with Euclidean, Nil or Sol geometry. In particular, $M$ is finitely covered by a torus bundle by \cite[Theorem 4.5]{EM}. Since any finitely generated subgroup $G$ of $\pi_1(M)$ is the fundamental group of a cover of $M$ with a compact core (\cite{Sco1}), it is easy to verify that $u(G) \leq 3$, contrary to our hypotheses (cf. \cite{EM}). Therefore $G$ must contain a rank $2$ free subgroup. 
\end{proof}

\section{$PSL(2,\C)$-character variety} 
\label{sec: psl2}

The goal of this section is to prove Theorem \ref{th:character} and Corollary \ref{cor:casson}. We begin with some definitions. 

The action of $SL_2(\mathbb C)$ on $\mathbb C^2$ descends to one of $PSL_2(\mathbb C)$ on $\mathbb CP^1$. We call a representation  with values in $PSL_2(\mathbb C)$ {\it irreducible} if the associated action on $\mathbb CP^1$ is fixed point free and {\it strictly irreducible} if the action has no invariant subset in $\mathbb CP^1$ with fewer than three points.

\begin{proof}[Proof of Theorem \ref{th:character}] 
If $G$ admits a representation with values in $PSL(2,\mathbb C)$ and with infinite image, then $\widehat G$ is infinite as finitely generated linear groups are residually finite. (See, for example, \cite[Proposition III.7.11]{LS}.)  
Let us assume then that the image of each representation of $G$ in $PSL(2,\mathbb C)$ is finite. We argue by contradiction assuming that the profinite completion $\widehat{G}$ is finite. 

The finiteness of $\widehat{G}$ implies that there is a smallest finite index subgroup $G_0$ of $G$ which is necessarily normal and contained in every finite index subgroup of $G$. Thus the kernel of each representation 
$\rho: G \to PSL(2,\mathbb C)$ contains $G_0$. Hence $\rho$ factors through $F_0  = G/G_0$. Since $G_0$ is perfect,  the finite quotient group $F_0 = G/G_0$ has periodic cohomology with period  $4$ by Claim \ref{claim:periodic}. The classification of  finite groups with cohomological period $2$ or $4$ implies that $F_0$ is of  tetrahedral type (III), octahedral type (IV), or icosahedral type (V) in \cite[Table 2]{GNS}. Up to conjugacy, such a group admits at most two irreducible representations with values in $PSL(2,\C)$, 
see \cite[Lemma 5.3]{BZ}, \cite{Wol}, contrary to the hypothesis. This proves the first assertion of Theorem \ref{th:character}.

To prove the second, suppose that $G$ is almost coherent. Then by Corollary \ref{cor:profinite}, $G$ either contains a rank $2$ free group or is virtually solvable. In the latter case, we argue as in the last paragraph of \S \ref{sec: r2free} to see that $M$ is geometric with Euclidian, Nil of Sol geometry. In the first two cases, $M$ is Seifert fibred with Euclidean base orbifold while in the third, $M$ is either the union of two twisted $I$-bundles over the Klein bottle or a torus bundle over the circle (\cite[Theorem 5.3(i)]{Sco2}). Fix a strictly irreducible homomorphism $\rho: \pi_1(M) \to PSL(2, \mathbb C)$. 

First suppose that $M$ is Seifert fibred with Euclidean base orbifold $\mathcal{B}$ and let $h \in \pi_1(M)$ be the class of the Seifert fibre. The proof of \cite[Lemma 3.1]{BeBo} is easily modified to show that $\rho(h) = \pm I$. Thus $\rho$ induces a homomorphism $\bar \rho: \pi_1(\mathcal{B}) \to PSL(2, \mathbb C)$. By hypothesis, there are at least three different characters of such $\bar \rho$, so \cite[Lemma 10.1]{BCSZ} implies that $M$ does not have base orbifold $S^2(2,3,6), S^2(2,4,4)$, or $S^2(3,3,3)$. Similarly $\mathcal{B}$ cannot be a torus. Suppose then that $\mathcal{B}$ is one of $S^2(2,2,2,2), P^2(2,2)$, or the Klein bottle $K$. Then $\pi_1(\mathcal{B})$ contains an infinite cyclic index $2$ normal subgroup which, as in the proof of \cite[Lemma 3.1]{BeBo}, lies in the kernel of $\bar \rho$. But then $\bar \rho$ factors through $\mathbb Z/2 \mathbb Z * \mathbb Z/2\mathbb Z$ and so cannot be strictly irreducible, a contradiction. Thus $M$ cannot be Seifert fibred. 

If $M$ is Sol, \cite[Proposition 7.1]{BCSZ} implies that the normal $\mathbb Z \oplus \mathbb Z$ subgroup of $\pi_1(M)$ corresponding to the torus (semi)fibre $T$ is sent by $\rho$ to $A = \mathbb Z/2 \mathbb Z \oplus \mathbb Z/2\mathbb Z \subset PSL(2, \mathbb C)$ and that the image of $\rho$ is $T_{12}$ in the fibre case and $O_{24}$ in the semifibre case. The reader will verify that the restriction of $\rho$ to $\pi_1(T)$ is unique up to conjugation in $PSL(2, \mathbb C)$ and we assume that all strictly irreducible representations of $\pi_1(M)$ to $PSL(2, \mathbb C)$ coincide on $\mathbb Z \oplus \mathbb Z$. 

First suppose that $M$ is a torus bundle with fibre $T$. Then the image of $\rho$ is $T_{12}$. Since the order of a non-trivial element of $T_{12}$ is $2$ or $3$ and the elements of order $2$ generate $A$, it follows that if $t \in \pi_1(M) \setminus \mathbb Z \oplus \mathbb Z$, then $\rho(t)$ has order $3$. On the other hand, under the conjugation, there are two $A$-orbits of elements of order $3$ in $T_{12}$. Thus there are at most two choices for the character of $\rho$ in the fibre case, a contradiction.

Next suppose that $M$ is a torus semibundle with semifibre $T$ and $\rho: \pi_1(M) \to O_{24} \subset PSL(2, \mathbb C)$ is an epimorphism. Then there is a $2$-fold cover $M_1 \to M$ such that $M_1$ is a torus bundle over the circle and the restriction of $\rho$ determines an epimorphism $\rho_1: \pi_1(M_1) \to T_{12}$. We claim that if $\rho': \pi_1(M) \to PSL(2, \mathbb C)$ is a homomorphism for which $\rho'|_{\pi_1(M_1)} = \rho|_{\pi_1(M_1)}$, then $\rho' = \rho$. Indeed, we must have $\rho'(\pi_1(M)) = O_{24}$  
and if $x \in \pi_1(M) \setminus \pi_1(M_1)$, conjugation by $\rho_1(x)$ on $T_{12}$ coincides with conjugation by $\rho_2(x)$, which implies that $\rho_1(x) = \rho_2(x)$ since $T_{12}$ is a strictly irreducible subgroup of $PSL(2, \mathbb C)$. But then as the previous paragraph implies that up to conjugation in $PSL(2, \mathbb C)$, there are at most two epimorphisms $\pi_1(M_1) \to T_{12}$, the same is true for epimorphisms $\pi_1(M) \to O_{24}$, a contradiction. This completes the proof of the second assertion of Theorem \ref{th:character}.
\end{proof}

\begin{proof}[Proof of Corollary \ref{cor:casson}]

Let $M$ be an integral homology sphere. If $\pi_1(M)$ is a non-trivial free product, then $\pi_1(M)$ contains a non-abelian free group (see for example \cite[Theorem 2]{Ly}), so we can assume that $\pi_1(M)$ is indecomposable. Hence $\pi_1(M)$ is the fundamental group of an irreducible integral homology sphere. Thus we assume, without loss of generality,  that $M$ is irreducible. 

If $\pi_1(M)$ is finite, the universal cover of $M$ is a closed homotopy 3-sphere on which $\pi_1(M)$ acts freely. Then $\pi_1(M)$ has cohomological period $4$ by \cite[Thm 4.8] {Swa}. The classification of these groups (\cite[Table 2]{GNS}) shows that it is isomorphic to the binary icosahedral group $SL(2, F_5)$, which admits, up to conjugacy, at most two irreducible representations with values in $PSL(2, \mathbb{C})$. The proof of \cite[Lemma 9.1]{Rez} then shows that $|\lambda(M)| \leq 1$, which contradicts our hypotheses. Therefore $\pi_1(M)$ must be infinite. Since $M$ is irreducible, $\pi_2(M) = \{0\}$ by the sphere theorem \cite{Sta2}. Furthermore, since $\widetilde{M}$ is non-compact, $0 = H_3(\widetilde{M}) = \pi_3(\widetilde{M}) = \pi_3(M)$ by the Hurewicz's theorem. It follows that $M$ is aspherical and thus $\pi_1(M)$ is a $PD(3)$ group which is, moreover, coherent (\cite{Sco1}).

\begin{lemma} \label{lem:casson} 
Let $M$ be an irreducible integral homology sphere with infinite fundamental group. Then one of the following properties holds.  

$(1)$ $\pi_1(M)$ contains a non-abelian free group.

$(2)$ Every non-trivial representation $\rho : \pi_1(M) \to SU(2)$ is finite with $\rho(\pi_1(M)) \cong SL(2, F_5)$ and the number of $SU(2)$-conjugacy classes of such representations 
is at least $2 \vert \lambda(M) \vert$.
\end{lemma}

\begin{proof} If $\pi_1(M)$ admits a representation $\rho : \pi_1(M) \to SU(2)$ with infinite image, then $\rho(\pi_1(M))$ cannot be virtually solvable as otherwise it would have positive first Betti number, contrary to the fact that $\pi_1(M)$ is perfect. The Tits alternative for linear groups (\cite{Ti}) then implies that $\rho(\pi_1(M))$ contains a non-abelian free group $F$. Let $F_0  \subset F$ be a rank $2$ free subgroup of $F$ and $H$ any subgroup of $\pi_1(M)$ generated by two elements which map to the the generators of $F_0$. Then $H$ is clearly free and non-abelian, which completes the proof of the lemma in this case. 

Assume then that every non-trivial representation $\rho : \pi_1(M) \to SU(2)$ has finite image. Since $\pi_1(M)$ is perfect, $\rho(\pi_1(M)) \cong SL(2, F_5)$ (cf. \cite{Wol}). It then follows from the definition of the Casson invariant (cf. \cite{AK}) and of the proof of \cite[Lemma 9.1]{Rez} that either the number of conjugacy classes of $SU(2)$-representations of $\pi_1(M)$ is at least 
$2 \vert \lambda(M) \vert$ or $M$ has a positive virtual first Betti number. In the latter case there is a finite cover $M_1$ of $M$ which contains a non-separating orientable $\pi_1$-injective surface $S$. This surface cannot be a $2$-sphere since $M_1$ is irreducible. Hence $\pi_1(S)$ contains a rank $2$ free group or is isomorphic to $\mathbb{Z} \oplus \mathbb{Z}$. In either case, $\pi_1(M)$ is not properly locally cyclic and since it has zero first Betti number, it cannot be virtually solvable. Thus Theorem \ref{th:main} implies that $\pi_1(M)$ contains a non-abelian free group. 
\end{proof}
Since irreducible representations with values in $SU(2)$ are conjugate in $SU(2)$ if and only if they are conjugate in $SL(2, \mathbb{C})$, Corollary \ref{cor:casson} follows from Lemma \ref{lem:casson}. 
\end{proof}

\end{document}